\documentclass[reqno,12pt]{amsart}

\usepackage{amsmath,amsthm,amssymb}
\newtheorem{theorem}{Theorem}[section]
\newtheorem{lemma}[theorem]{Lemma}
\newtheorem{corollary}[theorem]{Corollary}
\newtheorem{remark}[theorem]{Remark}

\theoremstyle{definition}

\numberwithin{equation}{section}

\begin{document}

\centerline{{\Large {\sc
On $L^1$-estimates for probability solutions}}}

\centerline{{\Large{\sc  to Fokker--Planck--Kolmogorov equations}}}

\vskip .1in

\centerline{{\large{\sc Vladimir I. Bogachev,
 Svetlana N. Popova,}}}

\centerline{{\large{\sc
 and Stanislav V. Shaposhnikov}}}

\vskip .2in

\centerline{Abstract}

\vskip .1in

We prove two new results connected with elliptic
Fokker--Planck--Kolmogorov equations with drifts integrable
with respect to solutions. The first result answers negatively
a long-standing question and shows that a density of a probability measure
satisfying the Fokker--Planck--Kolmogorov equation with a drift
integrable with respect to this density can fail to belong to the Sobolev
class~$W^{1,1}(\mathbb{R}^d)$.
There is also a version of this result for densities with respect
to Gaussian measures.
The second new result gives some positive information
about properties of such solutions: the solution density is proved
to belong to certain fractional Sobolev classes.



\vskip .1in

This research was supported by the Russian Science Foundation Grant 17-11-01058 (at Lomonosov
Mos\-cow State University).

\section{Introduction}

In the recent years, there has been a growing interest to various $L^1$-estimates for second order
partial differential operators, see, e.g.,
\cite{AN},
\cite{BB},
\cite{BrVS},
\cite{FM},
\cite{HT},
\cite{Maz},
\cite{PRT},
and \cite{VS1}--\cite{VS4},
where additional references can be found.
Their main feature is
that classical $L^p$-estimates for solutions to second order elliptic
equations valid for $p>1$ do not  extend directly to the case $p=1$
(see \cite{Orn}, \cite{DFT}, \cite{LM}, \cite{CFM}, \cite{BB}, and~\cite[Example~7.5]{GiaM}).
Some concrete  examples are mentioned below.
In particular, for the solution $f$ to the Poisson equation
$$
\Delta f=g, \quad f|_{\partial B}=0
$$
on the unit  ball $B$ in $\mathbb{R}^d$ one has
$$
\|f\|_{W^{p,2}(B)}\le C(p,d) \|g\|_{L^p(B)}
$$
with some number $C(p,d)$ provided that $p>1$,
but there is no such estimate for $p=1$ if $d>1$.
Here and throughout we use the symbol $W^{p,k}(\Omega)$
to denote the Sobolev space of functions on a domain $\Omega\subset \mathbb{R}^d$ that belong
to $L^p(\Omega)$ along with their partial derivatives up to order~$k$;
the Sobolev norm $\|f\|_{W^{p,k}}$ is the sum of the $L^p$-norms of the function $f$ and its
partial derivatives up to order~$k$. By $W^{p,k}_{loc}$ we denote the class
of functions $f$ such that $\zeta f\in W^{p,k}(\mathbb{R}^d)$ for all
functions $\zeta$ from the class $C_0^\infty(\mathbb{R}^d)$ of smooth compactly
supported functions.

The failure of an $L^1$-estimate of this kind is connected
with the fact that a solution $f$ for some $g\in L^1(B)$ does not belong
to the second Sobolev class $W^{1,2}(B)$. As a consequence,
a solution to the equation
$$
\Delta u={\rm div}\, v
$$
with a vector field $v$ of class $L^1$
can fail to belong to the first Sobolev class $W^{1,1}(B)$ and in case of smooth solutions
there is no estimate of $\|\nabla u\|_{L^{1}}$ through $\|v\|_{L^1}$. For example,
one can consider $u=\partial_{x_1}f$ and $v=(\partial_{x_1}g,0,\ldots)$ for $f$ and $g$
satisfying the first equation.

Questions of this type arise also for solutions to
Fokker--Planck--Kolmogorov equations on the whole space, which is the subject of this paper
and which has not been studied so far. A bounded Borel measure $\mu$ on $\mathbb{R}^d$ is said
to satisfy the Fokker--Planck--Kolmogorov equation
\begin{equation}\label{ek1}
L_b^{*}\mu=0
\end{equation}
with a Borel vector field $b$ locally integrable with respect to $\mu$ if
for the operator
$$
L_b \varphi=\Delta \varphi +\langle b,\nabla\varphi\rangle
$$
we have the identity
$$
\int_{\mathbb{R}^d} L_b\varphi(x)\, \mu(dx)=0\quad \forall\, \varphi\in C_0^\infty(\mathbb{R}^d).
$$
It is known (see \cite{BKR01} or~\cite{BKRS})
that in this case the measure $\mu$ has a density $\varrho$ with respect to Lebesgue measure and
$$
\Delta \varrho -{\rm div}\, (\varrho b)=0
$$
in the sense of the integral identity
\begin{equation}\label{id1}
\int_{\mathbb{R}^d} [\Delta \varphi +\langle b, \nabla \varphi\rangle ] \varrho\, dx=0
\quad \forall \varphi\in C_0^\infty(\mathbb{R}^d).
\end{equation}
Moreover,
if $|b|$ is locally integrable to some power $p>d$
with respect  to Lebesgue measure or with respect to~$\mu$, then $\varrho\in W^{p,1}_{loc}(\mathbb{R}^d)$
(although $\varrho$ can fail to be in the second Sobolev class $W^{p,2}_{loc}(\mathbb{R}^d)$
unlike the case of non-divergence form equations).
This is not true for $p<d$, but in case of the global integrability
the following fact holds for $p=2$
(see \cite{BKRS}, \cite{BR95}, and~\cite{BKR96}).
Suppose that $\mu=\varrho\, dx$ is a probability measure on $\mathbb{R}^d$ satisfying
equation~(\ref{ek1}), where, in addition,
$$
\int_{\mathbb{R}^d} |b|^2 \, d\mu=
\int_{\mathbb{R}^d} |b|^2 \varrho\, dx <\infty.
$$
Then $\varrho\in W^{1,1}_{loc}(\mathbb{R}^d)$ and the logarithmic gradient $\nabla\varrho/\varrho$ belongs to
the weighted space $L^2(\varrho\, dx)$
and one has a dimension-free estimate
\begin{equation}\label{estL2}
\bigl\| |\nabla\varrho/\varrho|\bigr\|_{L^2(\mu)}^2=
\int_{\mathbb{R}^d} \Bigl|\frac{\nabla\varrho}{\varrho}\Bigr|^2 \varrho\, dx\le
\int_{\mathbb{R}^d} |b|^2 \varrho\, dx
=\bigl\| |b|\bigr\|_{L^2(\mu)}^2.
\end{equation}
To be more precise, $\nabla\varrho/\varrho$ is the orthogonal projection of $b$
onto the closure of gradients of smooth compactly supported functions
in the Hilbert space $L^2(\varrho\, dx,\mathbb{R}^d)$ of $\mathbb{R}^d$-valued mappings.
Note that a local version of this result fails: see \cite[Example~1.6.10]{BKRS}.
It is obvious that it is not valid  for signed solutions.

There are also some sufficient conditions for membership of
$|\nabla\varrho/\varrho|$ in $L^p(\varrho\, dx)$ with $p>2$.
However, such conditions are not of the same
form as in case $p=2$ and require additional assumptions such as
a certain rate of convergence of $\langle b(x),x\rangle$ to~$-\infty$ as~$|x|\to+\infty$
(see \cite{MPR}, \cite{BKR06}, \cite{BKR09},
and~\cite{BKRS}).

It is still unknown whether
there are $L^p$-analogs of the above estimate for $p\not=2$. One goal of this paper is to show
that there is no such estimate for $p=1$. We actually show that
there is a sequence of smooth probability densities $\varrho_n$ satisfying the equations
$$
\Delta \varrho_n -{\rm div}\, (\varrho_n b_n)=0
$$
on $\mathbb{R}^2$ with smooth mappings $b_n$ such that
$$
\int_{\mathbb{R}^2} |\nabla\varrho_n|\, dx\ge n\int_{\mathbb{R}^2} |b_n|\, \varrho_n\, dx.
$$
We also show that there is a smooth probability solution $\varrho$
to the equation with smooth $b$ such that
 $|b|\varrho\in L^1(\mathbb{R}^2)$  and $|\nabla\varrho|$ is not
in $L^1(\mathbb{R}^2)$, so that $\varrho\not\in W^{1,1}(\mathbb{R}^2)$.
We emphasize that the difficulty concerned probability solutions, there was no problem with signed
solutions.

The assumption that $|b|$ is integrable with the weight $\varrho$ rather than
with respect to Lebesgue measure is quite natural. For example, any probability measure
with a density $\varrho\in W^{1,1}_{loc}(\mathbb{R}^d)$ satisfies
the equation $L_b^{*}\mu=0$ with $b=\nabla\varrho/\varrho$, where
we let $\nabla\varrho/\varrho:=0$ on the set $\{\varrho=0\}$. Obviously,
such $b$ can be very singular with respect to Lebesgue measure, but with
weight $\varrho$ it is locally integrable, and if
$\varrho\in W^{1,1}(\mathbb{R}^d)$, then $|b|\varrho\in L^1(\mathbb{R}^d)$.

The situation is similar with $L^1$-estimates with respect to Gaussian measures.
It is known (see \cite{BR95} or \cite{BKR96})
that if a Borel probability measure $\mu$ on $\mathbb{R}^d$
 satisfies equation (\ref{ek1})
with a drift
$$
b(x)=-x+v(x),
$$
where $x_i, |v|\in L^2(\mu)$
then $\mu$ has a density $f$ with respect to the standard Gaussian
measure $\gamma$ on $\mathbb{R}^d$, this density is in $W^{1,1}_{loc}(\mathbb{R}^d)$,
and the mapping $\nabla f/f$ is the orthogonal projection of $v$ to the closure
of the gradients of smooth compactly supported functions taken in
the Hilbert space $L^2(\mu,\mathbb{R}^d)$ of $\mathbb{R}^d$-valued mappings.
Hence
$$
\int_{\mathbb{R}^d}
 \Bigl|\frac{\nabla f}{f}\Bigr|^2\, d\mu
=\int_{\mathbb{R}^d}
 \frac{|\nabla f|^2}{f}\, d\gamma\le \int_{\mathbb{R}^d} |v|^2 \, d\mu.
$$
Our next result says that there is no such estimate for the $L^1$-norm of $|\nabla f|$, namely,
 there is a sequence of smooth vector fields $v_n$ on $\mathbb{R}^2$
such that
$$
\int_{\mathbb{R}^2} |\nabla f_n|\, d\gamma\ge n\int_{\mathbb{R}^2} |v_n|f_n\, d\gamma.
$$

It is instructive to consider the following formal manipulations.
For  locally Sobolev  $\varrho$ one can write (\ref{id1}) as
\begin{equation}\label{id2}
\int_{\mathbb{R}^d} \langle \nabla\varrho, \nabla \varphi\rangle \, dx=
\int_{\mathbb{R}^d} \langle \varrho b, \nabla \varphi\rangle \, dx
\quad \forall \varphi\in C_0^\infty(\mathbb{R}^d).
\end{equation}
If we substitute $\varphi=\log\varrho$, then we obtain
$$
\int_{\mathbb{R}^d} \frac{|\nabla \varrho|^2}{\varrho}\, dx=
\int_{\mathbb{R}^d} \langle b, \nabla \varrho\rangle \, dx
\le
\biggl(\int_{\mathbb{R}^d} \frac{|\nabla \varrho|^2}{\varrho}\, dx\biggr)^{1/2}
\biggl(\int_{\mathbb{R}^d} |b|^2\varrho\, dx\biggr)^{1/2},
$$
which yields (\ref{estL2}). It turns out that this manipulation can be
justified in case of $\mathbb{R}^d$ and in case of connected Riemannian manifolds
with certain curvature conditions (see~\cite{BRW}),
but not in case of arbitrary
connected manifolds. The latter is indeed impossible even in case $b=0$
because  that would mean the absence of nonzero integrable nonnegative
harmonic functions while examples of such functions are
known from \cite{Chung}, \cite{Grig}, and~\cite{LiSch}.
Let now  $\varphi$ be such that $\nabla\varphi=\nabla\varrho/|\nabla \varrho|$.
This would give the bound
$$
\int_{\mathbb{R}^d} |\nabla\varrho|\, dx\le \int_{\mathbb{R}^d} |b|\varrho\, dx,
$$
which, as we show below, is false even on $\mathbb{R}^2$.
Certainly, both substitutions are illegal, but the first one leads to a
correct conclusion. It would be interesting to find conditions under which
the second one can be also justified.

Our construction is based on a thorough study and certain modification of
the known old result of Ornstein \cite{Orn}, who
showed that there are smooth functions $g_n$ with support in a square in $\mathbb{R}^2$ such that
$$
\int_{\mathbb{R}^2} |\partial_x\partial_y g_n(x,y)|\, dx\, dy
\ge n \int_{\mathbb{R}^2} \Bigl[|\partial_x^2g_n(x,y)|+|\partial_y^2g_n(x,y)|\Bigr]\, dx\, dy.
$$
This result shows that the $L^1$-norm of the mixed derivative is not controlled
by the $L^1$-norm of the Laplacian. In particular, the Sobolev norm in the second class $W^{1,2}$
is not controlled by the $L^1$-norm of the Laplacian. The latter effect is much easier seen
by example of radial functions, as noted, e.g.,
in \cite{DFT} and~\cite{LM}. One can show that the function $f$ that is $\log\log r$ near zero
in $\mathbb{R}^2$ is not in the class $W^{1,2}$, but $\Delta f$ in the sense of distributions
is given by the usual pointwise
expression for the Laplacian outside of the origin and is integrable.
However, for our purposes this elementary example (actually, any radial function) is not enough,
as explained in~Remark~\ref{rem2.2}.
This is why
 we need a modification of Ornstein's result with a Lipschitz function and some
 other bound (see~(\ref{t2.1})). Moreover,  taking into account that in the
original example in \cite{Orn} some important technical
details of justification are omitted, we have to reproduce the whole example from
that paper with all details and  verification of some additional properties.
This is done in the last section. Moreover, we explain there how the desired
modification can be derived from Ornstein's result (but this reasoning does not provide
the missing details in Ornstein's construction).

Our positive result presented in Section~4 says that the Sobolev class $W^{1,1}$
to which solutions can fail to belong is actually the border line and that the integrability
of $|b|$ with respect to the measure $\varrho\, dx$ yields that $\varrho$ belongs
to fractional Sobolev classes of order of differentiability as close to $1$ as we wish.

\section{A modification of Ornstein's example}

Here we present  a modification of Ornstein's result that differs from his
original result by extra terms in the inequality.
These extra terms are needed in the case of the Fokker--Planck--Kolmogorov equation.
Its justification is postponed until the last section, since
it is rather involved technically, although the construction follows Ornstein's method.
In this and the last sections vectors in $\mathbb{R}^2$ are denoted by $(x,y)$
unlike the rest of the paper where single letters like $x$ are use
to denote vectors.

\begin{theorem}\label{t2.1}
For each $\delta \in (0, 1)$ there is a function
 $g_\delta \in C^{\infty}_0([-1, 1]^2])$ such that
\begin{equation}\label{e2.1}
\|\partial_x \partial_y g_\delta\|_1
\ge \frac{1}{\delta} \Bigl(\|\partial_x^2 g_\delta\|_1
+ \|\partial_y^2 g_\delta\|_1 +
\|\partial_x g_\delta\|_{\infty}
+ \|\partial_y g_\delta\|_{\infty} \Bigr).
\end{equation}
In addition, there is a Lipschitz {\rm(}even of class $C^1${\rm)} function $f$ that
vanishes outside of  $[-1, 1]^2$ such that there exist repeated Sobolev derivatives
with
$$
\partial_x^2 f, \ \partial_y^2 f \in L^1,
\ \hbox{ but } \
\partial_x \partial_y f \notin L^1,
$$
where
$\partial_x\partial_y$ is taken in the sense of distributions.
\end{theorem}

\begin{corollary}
There exist a probability density $\varrho\in C^\infty (\mathbb{R}^2)$
and a $C^\infty$-mapping $v\colon\, \mathbb{R}^2\to\mathbb{R}^2$
 such that $|v|\in L^1(\mathbb{R}^2)$ and
 $\Delta\varrho -{\rm div}\, v=0$, but $|\nabla \varrho|$
 does not belong to~$L^1(\mathbb{R}^2)$.
\end{corollary}
\begin{proof}
It is clear from the theorem that
for every $n$ one can find a function $g_n\in C^\infty_0([-1,1]^2)$ such that
$$
\|\partial_x\partial_y g_n\|_1=1, \
\|\partial_x^2 g_n\|_1+\|\partial_y^2 g_n\|_1\le 1/n, \
\|\partial_x g_n\|_\infty\le 1/n, \
\|\partial_y g_n\|_\infty\le 1/n.
$$
For the function $f_n=\partial_y g_n$ we have
$\|\partial_x f_n\|_1=1$, $\|f_n\|_\infty\le 1/n$, and
$\Delta f_n={\rm div}\, v_n$, where $v_n=(0, \Delta g_n)$,
so $\|v_n\|_1\le 1/n$. The function $f_n$  need not be nonnegative,
but $|f_n|\le 1/n$. We now consider $f_n$ on the square $[-2,2]^2$ and
find a bump function $u_n\in C_0^\infty([-2,2]^2)$ such that $0\le u_n\le 1/n$,
$u_n=1/n$ on  $[-1,1]^2$ and $|\nabla u_n|\le 2/n$. The function
$w_n=f_n+u_n\in C_0^\infty([-2,2]^2)$ is nonnegative, bounded by $2/n$ and
$\Delta w_n= {\rm div}\, (v_n+\nabla u_n)$, where $\|v_n+\nabla u_n\|_1\le 3/n$.

By using shifts we can find such functions $w_n$
with  supports in disjoint squares. Then the function
$w=\sum_{n=1}^\infty n^{-1}w_n$ is infinitely differentiable, nonnegative, and
$\Delta w={\rm div}\, v$, where $v=\sum_{n=1}^\infty n^{-1}(v_n+\nabla u_n)$,
$\|v\|_1\le 3\sum_{n=1}^\infty n^{-2}<\infty$, but
$\partial_x w$ does not belong to~$L^1(\mathbb{R}^2)$.
Finally, $\|w\|_1\le 2 \sum_{n=1}^\infty n^{-2}$, so multiplying $w$ by a constant
we obtain a probability density.
\end{proof}

\begin{remark}\label{rem2.1}
\rm
In the theorem and in the corollary, one can take the corresponding functions
such that the support is $[-1,1]^2$ or the whole plane, respectively. It suffices
to add to the constructed solution a smooth nonnegative function with the desired
support.
\end{remark}

\begin{remark}\label{rem2.2}
\rm
Let us explain why we could not  use a much simpler example of
the function $f(x,y)=\log\log r$ on $\mathbb{R}^2$ not belonging
to the second Sobolev class on the unit disc and satisfying the equation $\Delta f=g$
with $g$ integrable near the origin and also leading to the equation
$\Delta \partial_x f={\rm div}\, (g,0)$ whose solution is not in~$W^{1,1}$.
The point is that we need a probability solution
for the latter equation, but
if $f(x)=V(r)$ is an integrable radial function on the unit disc with integrable $V'(r)$ near zero on the real line
such that $\Delta f$ is integrable near zero in the plane, then,
recalling that $\Delta f$ in polar coordinates is
$\Delta f=V''(r)+r^{-1}V'(r)$, we see that $V''(r)$ must be integrable near zero in the plane
(i.e., $V''(r)r$ is integrable near zero on the real line).
Hence all second order partial derivatives of $f$ are integrable  as well, so $f$ is in the second Sobolev class.
\end{remark}

\section{The Fokker--Planck--Kolmogorov equation and the Gaussian case}

We now apply the example described above to constructing some examples
with the Fokker--Planck-Kolmogorov equation. Let us explain at once why
such examples are impossible on the real line. The point is that in the one-dimensional
case we have the equation $\varrho''-(\varrho b)'=0$, hence $\varrho'- \varrho b=C$
for some constant~$C$. It follows that $\varrho$ has a locally absolutely continuous
version. Since $\varrho b$ is integrable and $\varrho$ cannot be
separated from zero as $|x|\to\infty$, the constant $C$ must be zero, hence
$\varrho'$ is integrable as well.

\begin{theorem}\label{t3.1}
There exist
a continuous probability density $\varrho$ with compact
support and a Borel vector field $b$ with compact support on $\mathbb{R}^2$
such that  $|b|\varrho\in L^1(\mathbb{R}^2)$ and $\Delta \varrho-{\rm div}\, (\varrho b)=0$,
but $\varrho$ does not belong to the Sobolev class~$W^{1,1}_{loc}$.

There exist also a probability density $\varrho\in C^\infty(\mathbb{R}^2)$ and
a $C^\infty$-vector field $b$ such that
$\Delta \varrho-{\rm div}\, (\varrho b)=0$ and $|b|\varrho\in L^1(\mathbb{R}^2)$,
but $|\nabla\varrho|$ does not belong to~$L^1(\mathbb{R}^2)$.
\end{theorem}
\begin{proof}
We know that there is a continuous probability density $w$ with compact
support in $\mathbb{R}^2$ satisfying the equation
$\Delta w={\rm div}\, v$ with a Borel vector field $v$ with compact support such that
$|v|\in L^1(\mathbb{R}^2)$ and $w$ does not belong to $W^{1,1}(\mathbb{R}^2)$.
We now write the same equation as
$$
\Delta w={\rm div}\, (wb), \quad b:=\frac{v}{w},
$$
where on the set $\{w=0\}$ we define $b$ by the zero value.
Obviously, $|b|\in L^1(w\, dx)$,
although now we can loose the Lebesgue integrability of~$b$, of course.

We now construct an example of a
 smooth probability density $\varrho$ satisfying the equation $L_b^*(\varrho\, dx)=0$
 with smooth $b$, but still not belonging to $W^{1,1}(\mathbb{R}^2)$.
 To this end, we return to the examples of the previous section and using also
 Remark~\ref{rem2.1} find smooth nonnegative functions $g_n$
 with support exactly $[0,1]^2$ and smooth vector fields $v_n$
 with support in $[0,1]^2$ such that
 $\Delta g_n={\rm div}\, v_n$,
 $\|\nabla g_n\|_1=n^{-1}$,
 $\|g_n\|_\infty\le n^{-2}$, $\|v_n\|_1\le n^{-2}$.
 It is obvious from our construction that we can ensure
 the bound $\|\nabla g_n\|_{L^1(D)}> (2n)^{-1}$ on the twice
 smaller square $D$ with the same center.
 Next we cover the whole plane by squares of unit length with vertices
 at the integer points and slightly increase the obtained squares
 in order to produce overlapping squares $B_n$ such that every point is
 contained in the interior of some of these larger squares.
 Now each $B_n$ has intersections with eight other squares.
 Translating our functions $g_n$ we can construct
smooth nonnegative functions $f_n$ with supports exactly $B_n$
and vector fields $u_n$ of class $C_0^\infty(B_n)$ such that
$\Delta f_n={\rm div}\, u_n$,
$\|u_n\|_{L^1(D_n)}> (2n)^{-1}$, where $D_n$ is the square of edge length $1/2$
with the same center as~$B_n$,
$\|f_n\|_\infty\le n^{-2}$, and~$\|u_n\|_1\le 2n^{-2}$.

 The purpose of making $B_n$ overlapping is that now the function
 $f=\sum_{n=1}^\infty f_n$ is positive (simple translations of $g_n$
 would give  a function vanishing on the edges). Clearly, this function
 is infinitely differentiable and satisfies the equation $\Delta f={\rm div}\, u$
 with $u=\sum_{n=1}^\infty u_n$, where $|u|_1\le \sum_{n=1}^\infty \|u_n\|_1
 \le 2\sum_{n=1}^\infty n^{-2}$. It is also obvious that $|\nabla f|$
 is not integrable over the plane, since already the integral over the union of~$D_n$
 diverges. Taking $b=v/f$ as above, we obtain a smooth vector field with
 $|b|f\in  L^1(\mathbb{R}^2)$ such that $\Delta f={\rm div}\, (fb)$.
It remains to normalize $f$ to obtain a probability density.
  \end{proof}

We now consider the connection between the two cases mentioned above, where
densities  are taken with respect to Lebesgue measure and with respect to the standard Gaussian
measure~$\gamma$ on the plane with density $\varrho_2(x)=(2\pi)^{-1}\exp(-|x|^2/2)$.
Suppose that a probability measure $\mu$ with a density
 $\varrho$ satisfies the equation $L_b^*\mu=0$ with a drift~$b$. Let us set
$$
f(x)=\varrho(x)/\varrho_2(x).
$$
Certainly, the same measure $\mu=f\cdot\gamma$ satisfies
the equation with the same drift written as
$-x+v(x)$, where $v(x):=b(x)+x$.
Therefore, once we use the aforementioned field $b$ that coincides with $-x$
outside of a compact set, we obtain $v$ with compact support, so that
its integrability with respect to Lebesgue measure is the same as the integrability
with respect  to the Gaussian measure.

\begin{theorem}
There exist a vector field $v$ on $\mathbb{R}^2$ with compact support
such that $|v|$ is integrable with respect to Lebesgue measure, hence
with respect to~$\gamma$, and a continuous probability density
$\varrho$ proportional to $\varrho_2$ outside of a ball
such that the measure $\mu$ with density $\varrho$ satisfies
the equation $L_b^*\mu=0$ with $b(x)=-x+v(x)$, where $|x|, |v|\in L^1(\mu)$,
but $\varrho$ does not belong $W^{1,1}_{loc}(\mathbb{R}^2)$.
\end{theorem}
\begin{proof}
Let us take the function $w\ge0$  and the vector field $v$ with
compact support considered in the proof of Theorem~\ref{t3.1},
where $\Delta w={\rm div}\, v$,
$|v|\in L^1(\mathbb{R}^2)$, and $w\not\in W^{1,1}(\mathbb{R}^2)$.
We take the density $w+\varrho_2$, which satisfies the equation
$$
\Delta (w+\varrho_2)={\rm div}\, (v-x\varrho_2)
={\rm div}\, ((w+\varrho_2)b)
$$
with the drift
$$
 b=\frac{v-x\varrho_2}{w+\varrho_2},
 $$
 which is locally Lebesgue integrable and
 $b(x)=-x$ outside of the support of~$w$, so $|b|(w+\varrho_2)\in L^1(\mathbb{R}^2)$.
 Again, $w+\varrho_2\not\in W^{1,1}(\mathbb{R}^2)$.
\end{proof}

It is worth noting that we have constructed above examples of two types
in which solutions to Fokker--Planck--Kolmogorov equations have no Sobolev
regularity. One example gives a density $\varrho$ with compact support and a drift $b$ with compact
support such that $|b|\varrho$ is integrable, but $|b|$ is not locally Lebesgue integrable.
The other one gives a positive density $\varrho$ and a locally Lebesgue integrable
drift $b$ such that $|b|\varrho$ is integrable on the plane. We have no examples in which
the probability density
$\varrho$ and the drift $b$ have compact support and
 $|b|$ and $|b|\varrho$ are both integrable.
If in the two-dimensional case $|b|$ is locally
integrable to power larger than~$2$, then $\varrho$ not only belongs to $W^{2,1}_{loc}$,
but also has a positive continuous version by Harnack's inequality (see
\cite{BKR01}, \cite{BKR09} or~\cite{BKRS}), so that it is impossible
to make its support compact.

\section{A positive result in the $L^1$-setting}

 Let us prove a positive result on fractional differentiability of solutions.
 Although this result actually follows from the facts presented in the recent
 book~\cite[Chapter~1]{BKRS},
 it is not explicitly formulated there for the whole space in case $p\le d$.
 For the definition of the Sobolev space $H^{p,s}(\mathbb{R}^d)$, see
 \cite[\S\,1.8.1]{BKRS}, \cite{DNPV} or \cite{Stein}; for example, one can set
 $$
 H^{p,s}(\mathbb{R}^d)=(I-\Delta)^{-s/2}(L^p(\mathbb{R}^d)),
 $$
 where the operator $(I-\Delta)^{-s/2}$ is applied in the sense of distributions.

 \begin{theorem}
 {\rm(i)}
 Suppose that $\mu$ is a bounded Borel measure on $\mathbb{R}^d$ satisfying the equation
 $L_b^*\mu=0$ with $|b|\in L^1(|\mu|)$. Then $\mu$ has a density $\varrho$ belonging to the
 fractional Sobolev class $H^{r,\alpha}(\mathbb{R}^d)$ for each
 $r>1$ and $\alpha< 1-d(r-1)/r$, where $1-d(r-1)/r>0$ whenever $1<r<d/(d-1)$. In particular,
 $\varrho\in L^s(\mathbb{R}^d)$ for each exponent $s\in [1,d/(d-1))$.

 {\rm(ii)} If $|b|\in L^p(|\mu|)$ with some $p\in (1,d]$, then $\varrho\in W^{q,1}(\mathbb{R}^d)$
 for each exponent $q<d/(d+1-p)$, hence $\varrho\in L^s(\mathbb{R}^d)$ for all $s<d/(d-p)$.
 \end{theorem}
 \begin{proof}
(i) We know that $\mu$ has a density $\varrho\in L^1(\mathbb{R}^d)$. Then
$$
\Delta\varrho-\varrho=g+{\rm div}\, F, \quad F=\varrho b,
$$
where $g:=-\varrho\in L^1(\mathbb{R}^d)$, $|F|\in L^1(\mathbb{R}^d)$. By the Sobolev embedding theorem,
the function $g$ and the components of $F$ belong to the negative Sobolev class
$W^{s,-1}(\mathbb{R}^d)$ with any $s<d'=d/(d-1)$. Therefore, $g, {\rm div}\, F\in W^{s,-2}(\mathbb{R}^d)$, which
yields that $\varrho\in L^{s}(\mathbb{R}^d)$. Moreover, we also have
$$
\Delta\varrho-\varrho\in H^{r,-1-d(r-1)/r-\varepsilon}(\mathbb{R}^d)
\quad \hbox{for all } \ r>1.
$$
Therefore,
$$
\varrho\in H^{r,1-d(r-1)/r-\varepsilon}(\mathbb{R}^d).
$$

(ii) In case $p>1$, by assertion (i) we have $\varrho\in L^s(\mathbb{R}^d)$ for any $s\in [1,d')$.
Let us take $\delta>0$ so small that
$$
1+\delta+\delta^2<d', \quad 1+\delta+\delta^2+\delta^3\le p.
$$
Set
$$
r:=1+\delta, \quad p_1:=1+\delta^2.
$$
By H\"older's inequality
$$
\int_{\mathbb{R}^d}
 |b|^{p_1}|\varrho|^{p_1}\, dx \le
 \biggl(\int_{\mathbb{R}^d} |b|^{p_1r}|\varrho|\, dx\biggr)^{1/r}
\biggl(\int_{\mathbb{R}^d} |\varrho|^{1+(p_1-1)r/(r-1)}\, dx\biggr)^{r/(r-1)}.
$$
By our choice $p_1r\le p$ and $1+(p-1)r/(r-1)=1+\delta(1+\delta)<d'$.
Hence the right-hand side of the previous estimate is finite. This yields the inclusion
$\varrho\in W^{p_1,1}(\mathbb{R}^d)$. Therefore, $\varrho\in L^{p_1d/(d-p_1)}(\mathbb{R}^d)$ by the
Sobolev embedding theorem. Now the same reasoning with iterations as in \cite[Theorem 1.8.2]{BKRS}
in the local case shows that we can raise the order of the Sobolev class for $\varrho$ as close to
$d/(d+1-p)$ as we wish.
 \end{proof}

 In case $p>d$ we have $\varrho\in W^{p,1}(\mathbb{R}^d)$ (see \cite[Chapter~1]{BKRS}),
 hence $\varrho$ has a
 bounded continuous version.

 However, even weaker assumptions are sufficient to increase the global order of integrability
 of $\varrho$. Namely, suppose that $\varrho$ is the density of a bounded measure $\mu$
 satisfying the equation $L_b^*\mu=0$
 with $b$ such that
 $$
\sup_{a\in \mathbb{R}^d} \int_{U(a)} |b(x)|\, |\varrho(x)|\, dx \le M,
 $$
 where $U(a)$ is the ball of radius $1$ centered at $a$. Then
 $\varrho\in L^p(\mathbb{R}^d)$ for every exponent $p\in [1,d/(d-1))$ and
 $$
 \|\varrho\|_{L^p}\le C(M,d,p)\|\varrho\|_{L^1},
 $$
 where $C(M,d,p)$ is a number depending only on $d,p,M$.
Indeed, we can assume that
$\|\varrho\|_{L^1}\le 1$.
It follows from the
local estimates established  in \cite[\S\,1.5]{BKRS} that there is a number $C_1(p,d,M)$ independent of
$\varrho$ such that
$$
\|\varrho\|_{L^p(B(a))}\le C_1(p,d,M) \|\varrho\|_{L^1(U(a))}
$$
for every ball $B(a)$ of radius $1/2$ centered at $a$. Since $\|\varrho\|_{L^1(U(a))}\le 1$, this yields the bound
$$
\|\varrho\|_{L^p(B(a))}^p\le C_1(p,d,M)^p \|\varrho\|_{L^1(U(a))}.
$$
Hence the integral of $|\varrho|^p$ over the whole space is estimated
by $C(d)C_1(p,d,M)^d$ with some number $C(d)$ depending only on $d$.

Note that a local version of the previous theorem is proved in
\cite{BKR97} (see also \cite[Chapter~1]{BKRS}) for nonconstant infinitely
differentiable matrix~$A$. The previous theorem can be also generalized to
nonconstant~$A$ provided that  the second order elliptic operator
$a^{ij}\partial_{x_i}\partial_{x_j} -1$ has the same properties
as the Laplacian in the scale of Sobolev spaces. For example, this is true
if $A(x)=A_0+A_1(x)$, where $A_0$ is a constant positive definite matrix
and $A_1$ has entries in $C_b^\infty(\mathbb{R}^d)$ and $\|A\|_\infty$ is
 sufficiently small.

Once the density $\varrho$ belongs to $H^{p,s}(\mathbb{R}^d)$, one can use known
embedding theorems for fractional Sobolev spaces (see, e.g.,~\cite{DNPV}),
in particular, there is a continuous embedding into $L^{p^{*}}(\mathbb{R}^d)$ with
$p^{*}=dp/(d-sp)$ if $sp<d$. There are also fractional Hardy inequalities
estimating integrals
of functions like $|f(x)|^q/{\rm dist}(x,\partial \Omega)^\alpha$
over a domain $\Omega$ via a suitable fractional Sobolev norm (see \cite{DF}).

For additional results on regularity of solutions in case of $A$ of low regularity,
see the recent paper~\cite{BS17}.

\section{Justification of Theorem~\ref{t2.1}}

Here we present a detailed justification of Theorem~\ref{t2.1}, which
is needed not only because some details have been omitted in~\cite{Orn},
but also because we need a bound with additional terms.
So it does not come as a surprise that our justification is twice longer
than in~\cite{Orn}, although we essentially follow Ornstein's construction and partly
use the same notation. On the other hand, we show below how a similar result
(which can be also used for our purposes) can be
deduced from Ornstein's  example (if we do not intend to provide all details
for the latter).

\begin{proof}[Proof of Theorem~\ref{t2.1}]
For any function $\psi$, let us set
 $\psi_x = \frac{\partial \psi}{\partial x}$,
 $$
 \psi^x(x_0, y_0) = \int_{-1}^{x_0} \psi(x, y_0)\, dx,
 $$
 and let  $Var_x \psi(y_0)$ be the variation of the function
  $s\mapsto \psi(s, y_0)$ on $[-1,1]$ for fixed~$y_0$.
  So if $\psi$ is smooth (but actually at this stage we construct piecewise
  constant functions), $Var_x \psi(y_0)$ is the integral
  of $|\partial_x \psi(s,y_0)|$ in $s$ over $[-1,1]$.
Similarly we define $\psi^y(x_0, y_0)$ and $Var_y \psi(x_0)$.

Let $\delta \in (0, 1)$ be a rational number.
We shall find $g_\delta$ in the form of the integral over $[-1,x]\times [-1,y]$ of a certain smoothing
of a suitable function~$p_n$ described below.
We construct a sequence of functions $p_n(x, y)$ on $[-1,1]^2$
with the following properties (all double integrals are taken over $[-1,1]^2$):

\begin{itemize}
\item[1)] For each $n$ there exist a partition of $[-1, 1]$ on the $y$-axis into
intervals
 $J_{n,i} = \{a_{n,i} \le y \le a_{n,i + 1}\}$ and a partition of the interval
  $[-1, 1]$ on the $x$-axis into intervals
  $L_{n,j} = \{b_{n,j} \le x \le b_{n,j + 1}\}$. The function $p_n$ is defined and constant
on every open rectangle
$(b_{n,j}, b_{n,j + 1}) \times (a_{n,i}, a_{n,i + 1})$, but is not defined on the
boundaries of the rectangles.

\item[2)] $p_n(x, y) = 0$ outside of $[-1, 1]^2$.

\item[3)] ${\displaystyle \int_{-1}^{1} p_n(x_0, y)\, dy = 0 \
\forall\, x_0\in [-1,1]}$,
${\displaystyle\int_{-1}^{1} p_n(x, y_0)\, dx = 0 \ \forall\, y_0\in [-1,1]}$.

\item[4)] ${\displaystyle\iint |p_n(x, y)|\, dx\, dy > C_1 \delta n}$, where $C_1$ is a constant
independent of $n$ and $\delta$.

\item[5)] ${\displaystyle\int_{-1}^1 Var_x p_n^y(y)\, dy < C_2\delta \iint |p_n(x, y)|\, dx\, dy}$,
where  $C_2$ is a constant independent of $n$ and $\delta$.

\item[6)] ${\displaystyle\int_{-1}^{1} Var_y p_n^x(x)\, dx = C_3}$, where $C_3$ is a constant
independent of $n$ and $\delta$. Note that here and in the previous
item we would have $\|\partial_y p_n^x\|_1$ and $\|\partial_x p_n^y\|_1$
for smooth functions.

\item[7)] For every $x_0\in [-1,1]$, the function $p_n^x(x_0, y)$
is nonincreasing piecewise constant on $(a_{n,2}, a_{n,N_n - 1})$,
 and whenever $2 < i < N_n$
one has
$$
p_n^x(x_0, y_i) - p_n^x(x_0, y_{i - 1}) \le 2^{-(n - 1)},
\quad y_i \in J_{n,i}, y_{i - 1} \in J_{n,i - 1}.
$$

\item[8)] $|p_n^x(x, y)| \le 1$, $|p_n^y(x, y)| \le \delta (2 - 2^{-(n - 1)})$
for all $x, y\in [-1,1]$.

\item[9)] $p_n(x, y) = -p_n(x, -y)$.
\end{itemize}

We now define the function $p_1$:
$$
p_1 = 1 \quad \mbox{if} \quad
(x, y) \in (-1/2, 0) \times (-\delta, -\delta/4) \cup (0, 1/2) \times (\delta/4, \delta),
$$
$$
p_1 = -1 \quad \mbox{if} \quad
(x, y) \in (0, 1/2) \times (-\delta, -\delta/4) \cup (-1/2, 0) \times (\delta/4, \delta),
$$
and $p_1 = 0$ else.

Then
$$\iint |p_1(x, y)|\, dx \,dy = 3\delta/2,
$$
$$
\int_{-1}^1 Var_x p_1^y(y)\, dy \le \widetilde{C}_1  \delta^2,
$$
$$
\int_{-1}^1 Var_y p_1^x(x)\, dx =: C_3,
$$
where $C_3$ is independent of $\delta$, because
$$
p_1^x(x,y)=0 \quad \hbox{outside of} \quad
[-1/2,1/2]\times ([-\delta, -\delta/4]\cup [\delta/4, \delta])
$$
and $p_1^x(x,y)=(1-2|x|){\rm sign}\, y$ else, so
$Var_y p_1^x(x)=0$ if $x\in [-1,-1/2]\cup [1/2,1]$,
$Var_y p_1^x(x)=-2|x|+1$ if $x\in (-1/2,1/2)$.
Hence  $p_1$ satisfies 1)--9).

Suppose that $p_n$ is defined and show how to define $p_{n + 1}$.
It suffices to define $p_{n + 1}$ for $y < 0$ and use 9) to extend to $y > 0$.
For every $i > 2$ we take the interval
 $\widetilde J_{n,i} = (a_{n,i} - \alpha/2, a_{n,i} + \alpha/2)$ of length
  $\alpha$ (where a rational number $\alpha$ will be chosen later).
  Outside of $(-1, 1) \times (\cup_{i} \widetilde J_{n,i})$ we let $p_{n + 1}=p_n$.
On $(-1, 1) \times \widetilde J_{n,i}$ we define $p_{n + 1}$
as the sum of two functions $r_1^i$ and $r_2^i$ (we omit $n$ in their notation), where
$$
r_1^i(x, y) = \frac{p_n(x, y_{i - 1}) + p_n(x, y_i)}{2},
\quad y\in \widetilde{J}_{n,i}, \ y_{i - 1} \in J_{n,i - 1}, \ y_i \in J_{n,i},
$$
outside of the strips $(-1, 1) \times \widetilde J_{n,i}$ we let $r_1^i=0$,
and $r_2^i$ is defined as follows.
We partition the strip
 $(-1, 1) \times \widetilde J_{n,i}$ into rectangles $K_k$ of height
  $\alpha$ and width $\alpha/\delta$ (again we suppress $n$ in this notation).
  Next, each rectangle $K_k$ is partitioned into four rectangles
of height $\alpha/2$ and width $\alpha/(2\delta)$ each.
Take $\beta^i_k$ such that
$$
\frac{\alpha}{2\delta} \beta^i_k
= \frac{p_n^x(x_k, y_{i - 1})
- p_n^x(x_k, y_i)}{4}, \quad y_{i - 1} \in J_{n,i  - 1}, y_i \in J_{n,i},
$$
where $x_k$ is the $x$-coordinate of the center of the rectangle $K_k$.
Define $r_2^i$ to equal $\beta_k^i$
 on the lower left and upper right rectangle of $K_k$, and let $r_2^i$ equal
 $-\beta_n^i$ on the remainder of~$K_k$.
 Outside of the strips $(-1, 1) \times \widetilde J_{n,i}$ we let $r_2^i=0$.

Let us verify that $p_{n + 1}$ satisfies 1)--9). Properties
1), 2), 9) are obvious. Property~3) follows from~9) and the fact that
$$
\int_{-1}^{1} r_1^i(x, y_0)\, dx = 0,
\quad
\int_{-1}^{1} r_2^i(x, y_0)\, dx = 0
\quad \forall\, y_0\in [-1,1].
$$

Let us prove that $p_{n + 1}$ satisfies 7).
It suffices to show that for all $i > 2$ and all $x$
\begin{multline}\label{eq5.1}
p_{n + 1}^x(x, y_{i - 1}) - p_{n+1}^x(x, y'),
p_{n+1}^x(x, y') - p_{n+1}^x(x, y''), p_{n+1}^x(x, y'') - p_{n + 1}^x(x, y_i)
\\
\in \Bigl[0, \frac{p_n^x(x, y_i) - p_n^x(x, y_{i - 1})}{2}\Bigr],
\end{multline}
where $y_{i - 1} \in (a_{n,i - 1}, a_{n,i} - \alpha/2)$,
$y' \in (a_{n,i} - \alpha/2, a_{n,i})$,
$y'' \in (a_{n,i}, a_{n,i} + \alpha/2)$,
$y_i \in (a_{n,i} + \alpha/2, a_{n,i + 1})$.
Since for every  $y_0$ the function
 $p_{n + 1}^x(x, y_0)$ is linear on the intervals
  $[0, \alpha/(2\delta)]$,
  $[\alpha/(2\delta), \alpha/\delta],\dots$, it suffices to verify (\ref{eq5.1})
for the endpoints $x \in \{\alpha m/(2\delta)\colon m \in \mathbb N\}$.
First we consider the endpoints of the form $x = \alpha m/\delta$, $m \in \mathbb N$.
In this case in our calculation of $p_{n+1}^x$ the $\beta$-terms mutually cancel, hence
$$
 p_{n+1}^x(x, y') =  p_{n+1}^x(x, y'') = \frac{p_n^x(x, y_i) + p_n^x(x, y_{i - 1})}{2},
$$
so (\ref{eq5.1}) is fulfilled.
The remaining endpoints are  the $x$-coordinates $x_k$ of the centers of the rectangles $K_k$
(i.e., equal $\alpha m/(2\delta)$ with odd~$m$). For them we have
$$
p_{n+1}^x(x_k, y') =
\frac{p_n^x(x_k, y_i) + p_n^x(x_k, y_{i - 1})}{2}
 + \frac{\alpha}{2\delta} \beta^i_k,
$$
$$
p_{n+1}^x(x_k, y'') =  \frac{p_n^x(x, y_i) + p_n^x(x, y_{i - 1})}{2} -
\frac{\alpha}{2\delta} \beta^i_k,
$$
and the definition of $\beta^i_k$ yields that
$$
p_{n + 1}^x(x_k, y_{i - 1}) -  p_{n+1}^x(x_k, y') =
 \frac{p_n^x(x_k, y_{i - 1}) - p_n^x(x_k, y_i)}{4},
$$
$$
p_{n + 1}^x(x_k, y') -  p_{n+1}^x(x_k, y'')
= \frac{p_n^x(x_k, y_{i - 1}) - p_n^x(x_k, y_i)}{2},
$$
$$
p_{n + 1}^x(x_k, y'') -  p_{n+1}^x(x_k, y_i)
= \frac{p_n^x(x_k, y_{i - 1}) - p_n^x(x_k, y_i)}{4}.
$$
Hence 7) is fulfilled.

We observe that  7) obviously yields 6): since
  $p_{n + 1}^x(x, y) = p_1^x(x, y)$ whenever
   $(x, y) \in (0, 1) \times (a_{n,1}, a_{n,N_n})$ and $p_{n + 1}^x(x, y)$
   is a nonincreasing function of $y$ on the interval
   $(a_{n,2}, a_{n,N_n - 1})$, one has
$Var_y p_{n + 1}^x(x) = Var_y p_1^x(x)$, whence we obtain~6).

Let us show that $p_{n + 1}$ satisfies 8).
Since  $p_{n + 1}^x(x, y) = p_1^x(x, y)$ whenever
 $(x, y) \in (0, 1) \times (a_{n,1}, a_{n,N_n})$ and $p_{n + 1}^x(x, y)$ is a nonincreasing
 function of $y$ on $(a_{n,2}, a_{n,N_n - 1})$,
 one has
 $$
 \max_{x, y} |p_{n + 1}^x(x, y)| \le \max_{x, y} |p_1^x(x, y)| \le 1.
 $$
Let us estimate
$\max_{x, y} |p_{n + 1}^y(x, y)|$. If $y \notin \cup \widetilde J_{n,i}$, then
$p_{n + 1}^y(x, y) = p_n^y(x, y)$, since
$$
\int_{\widetilde J_{n,i}} [r^1_i(x, y) + r^2_i(x, y)]\, dy
= \int_{\widetilde J_{n,i}} p_n(x, y)\, dy.
$$
If $y \in \widetilde J_{n,i}$, then
\begin{align*}
|p_{n + 1}^y(x, y)| &\le |p_{n}^y(x, y_{i - 1})|
+ \frac{\alpha}{2} \max_{k} \beta^i_k
\\
&= |p_{n}^y(x, y_{i - 1})| + \frac{\delta}{2}
 \max_{k} (p_n^x(x_k, y_{i - 1}) - p_n^x(x_k, y_i))
\\
&\le \delta (2 - 2^{-(n - 1)}) + \delta 2^{-n} = \delta (2 - 2^{-n}),
\end{align*}
where the last inequality follows from the fact that $p_n$ satisfies 8) and 7).

Let us show that $p_{n + 1}$ satisfies 5) for sufficiently small $\alpha$. We observe that
for every $i$ we have
$$
\int_{-1}^1 Var_x (r_2^i)^y(y)\, dy =
\widetilde{C}_2 \alpha^2 \sum_{k} \beta_k^i,
$$
where $\widetilde{C}_2$ is a  constant (independent of $n$ and $\delta$) and
$$
\iint |r_2^i| \, dx \, dy = \frac{\alpha^2}{\delta} \sum_{k} \beta_k^i.
$$
Hence
$$
\int_{-1}^1 Var_x (r_2^i)^y(y)\, dy =\widetilde{C}_2 \delta \iint |r_2^i|\, dx\, dy.
$$
Note that
$$
\lim_{\alpha \to 0} \int_{-1}^1 Var_x r_1^y(y)\, dy = 0,
\quad  \lim_{\alpha \to 0} \iint |r_1|\, dx\, dy = 0,
\quad \mbox{where} \quad r_1 = \sum_{i} r_1^i.
$$
Therefore, for sufficiently small $\alpha$ there holds the inequality
$$
\int_{\cup \widetilde J_{n,i}} Var_x p_{n + 1}^y(y)\, dy
< (\widetilde{C}_2+1) \delta \int_{\cup \widetilde J_{n,i}}
 \int_{-1}^{1} |p_{n + 1}(x, y)|\, dx dy.
$$
We have
$$
Var_x p_{n + 1}^y(y) =  Var_x p_{n}^y(y) \quad \mbox{if}
\quad y \not\in \cup \widetilde J_{n,i},
$$
because $p_{n+1}^y(x,y)=p_{n}^y(x,y)$ if $y\not\in \cup \widetilde J_{n,i}$.
Since $p_n$ satisfies 5), one has
$$
\int_{-1}^1 Var_x p_{n + 1}^y(x, y) \, dy <
\max(C_2, \widetilde{C}_2+1) \delta \iint |p_{n + 1}(x, y)|\, dx\, dy,
$$
whence we obtain that  $p_{n + 1}$ satisfies 5)
(from the very beginning we take $C_2 > \widetilde{C}_2+1$, which is possible, since
$\widetilde{C}_2$ is a universal constant independent of $n$ and~$\delta$).

Let us show that $p_{n + 1}$ satisfies  4) for sufficiently small $\alpha$.
We have
$$
\lim_{\alpha \to 0} \iint |p_{n + 1}| \, dx\, dy =
\iint |p_n|\, dx\, dy + \lim_{\alpha \to 0} \sum_{i}
\int_{\widetilde J_{n,i}} \int_{-1}^{1} |r^2_i(x, y)|\, dx\, dy,
$$
\begin{multline*}
\lim_{\alpha \to 0} \int_{\widetilde J_{n,i}} \int_{-1}^{1} |r^2_i(x, y)|\, dx\, dy =
\lim_{\alpha \to 0} \frac{\alpha^2}{\delta} \sum_{k} \beta^i_k
\\
= \lim_{\alpha \to 0} \frac{\alpha}{2} \sum_k (p_n^x(x_k, y_{i - 1}) - p_n^x(x_k, y_i))
 \\
= \frac{\delta}{2} \int_{-1}^{1} (p_n^x(x, y_{i - 1}) - p_n^x(x, y_i))\, dx,
\quad \mbox{where} \quad y_{i - 1} \in J_{n,i - 1}, y_i \in J_{n,i}.
\end{multline*}
The last equality is just the limit of the Riemann sums with partitions of length~$\alpha/\delta$.
Hence
$$
\lim_{\alpha \to 0} \iint |p_{n + 1}| \, dx\, dy
= \iint |p_n|\, dx\, dy + \frac{1}{2}\delta \int_{-1}^1 Var_y p_n^x(x)\, dx
> C_1\delta n + \frac{C_3}{2}\delta,
$$
since $p_n$ satisfies  4) and
 $$
 \int_{-1}^{1} Var_y p_n^x(x)\, dx =  C_3.
 $$
 We now take $C_1 < \min(C_3/2,3/2)$.

For $p_n$ we have
$$
\int_{-1}^1 Var_x p_n^y(y) \, dy
+ \int_{-1}^1 Var_y p_n^x(x)\, dx
+ \|p_n^x\|_{\infty} + \|p_n^y\|_{\infty} < C_2 \delta \|p_n\|_1 + C_3 + 1 + 2\delta,
$$
$$
\|p_n\|_1 > C_1 \delta n.
$$
Hence for sufficiently large $n$ we obtain
\begin{equation}\label{ek5.2}
\int_{-1}^1 Var_x p_n^y(y) \, dy + \int_{-1}^1 Var_y p_n^x(x) \, dx
+ \|p_n^x\|_{\infty} + \|p_n^y\|_{\infty} < C_2 \delta \|p_n\|_1.
\end{equation}
For each  $n$ we can smooth $p_n$
in the variable~$x$ as follows.
Let $\varrho  \in C^\infty(\mathbb R)$ be a probability density with support in
$[-1, 1]$. Let
$$
q_n(x, y) = \int_{-1}^1 p_n(x - t, y) \varrho _\varepsilon(t)\, dt,
\quad \mbox{where} \quad \varrho _{\varepsilon}(t)
= \frac{1}{\varepsilon}\varrho \Bigl(\frac{t}{\varepsilon}\Bigr),
\ \varepsilon>0.
$$
We do not indicate dependence of $q_n$ on $\varepsilon$ that will be taken
sufficiently small.
Then
 $\|q_n\|_1 \to \|p_n\|_1$ as $\varepsilon \to 0$.
 The functions  $q_n^x$ and $q_n^y$ satisfy the equalities
$$
q_n^x(x, y) = \int_{-1}^1 p_n^x(x - t, y)
\varrho _\varepsilon(t)\, dt, \quad q_n^y(x, y)
= \int_{-1}^1 p_n^y(x - t, y) \varrho _\varepsilon(t)\, dt.
$$
It follows that
$$
\|q_n^x\|_\infty \le  \|p_n^x\|_\infty,
\quad
\|q_n^y\|_\infty \le  \|p_n^y\|_\infty.
$$
Let us estimate
 $$
 \int_{-1}^1 Var_x q_n^y(y)\, dy \quad\hbox{and}\quad
 \int_{-1}^1 Var_y q_n^x(x)\, dx
 $$
 from above. We have
\begin{multline*}
Var_x q_n^y(y)
= \sup\Bigl\{\sum_i |q_n^y(c_i, y) - q_n^y(c_{i - 1}, y)|\colon\
 -1 \le c_1 \le \cdots \le c_n \le 1\Bigr\}
  \\
= \sup\biggl\{\sum_i \biggl|\int_{-1}^1 p_n^y(c_i - t, y) - p_n^y(c_{i - 1} - t, y)
\varrho _\varepsilon(t)\, dt\biggr| \biggr\}
\\
\le \sup\biggl\{\int_{-1}^1
\sum_i |p_n^y(c_i - t, y) - p_n^y(c_{i - 1} - t, y)|
\varrho _\varepsilon(t)\, dt \biggr\} \le Var_x p_n^y(y).
\end{multline*}
Therefore,
$$
\int_{-1}^1 Var_x q_n^y(y)\, dy \le \int_{-1}^1 Var_x p_n^y(y)\, dy,
$$
\begin{multline*}
Var_y q_n^x(x) = \sup
\Bigl\{\sum_i |q_n^x(x, c_i) - q_n^x(x, c_{i - 1})|\colon\
 -1 \le c_1 \le \cdots \le c_n \le 1\Bigr\}
 \\
= \sup\biggl\{\sum_i \biggl|\int_{-1}^1 p_n^x(x - t, c_i)
- p_n^x(x - t, c_{i - 1}) \varrho _\varepsilon(t)\, dt\biggr| \biggr\}
\\
\le \sup\biggl\{\int_{-1}^1 \sum_i
|p_n^x(x - t, c_i) - p_n^x(x - t, c_{i - 1})| \varrho _\varepsilon(t)\, dt \biggr\}
\le
\int_{-1}^1 Var_y p_n^x(x - t) \varrho _\varepsilon(t)\, dt,
\end{multline*}
which yields that
$$
\int_{-1}^1 Var_y q_n^x(x)\, dx \le \int_{-1}^1 Var_y p_n^x(x)\, dx.
$$
Therefore, for sufficiently small
 $\varepsilon$, for $q_n$ we have inequality (\ref{ek5.2}).

Let us show that $q_n$ has property 3) from the list for $p_n$, i.e., we have to show that
 $q_n^y(x, 1) = 0$ for all $x$ and $q_n^x(1, y) = 0$ for all~$y$.
 This is needed in order to ensure that $q_n^y$ and $q_n^x$ vanish outside of~$[-1,1]^2$.
The equality $q_n^y(x, 1) = 0$  follows from property~3) for $p_n$ and the fact that $p_n^y(x, 1) = 0$
for all~$x$. In addition,
\begin{align*}
q_n^x(1, y) &= \int_{-1}^1 q_n(x, y)\, dx
= \iint p_n(t, y) \varrho _\varepsilon(x - t)\, dt\, dx
 \\
&=\int_{-1}^1 p_n(t, y) \int_{-1}^1 \varrho_\varepsilon(x - t)\, dx\,  dt
= \int_{-1}^1 p_n(t, y)\, dt = 0.
\end{align*}
Similarly, smoothing the constructed function in the variable $y$,
 we obtain a function of class~$C_0^\infty$, again denoted by $q_n$, satisfying~3) from the list for~$p_n$
 and  inequality~(\ref{ek5.2}).

Let $g_\delta = (q_n^{x})^y$. Then $g_\delta \in C_0^\infty$ and
\begin{multline*}
\|\partial_x^2 g_\delta\|_1
+ \|\partial_y^2 g_\delta\|_1 +
\|\partial_x g_\delta\|_{\infty}
+ \|\partial_y g_\delta\|_{\infty}
= \|(q_n^y)_x\|_1 + \|(q_n^x)_y\|_1
+ \|q_n^x\|_{\infty} + \|q_n^y\|_{\infty}
\\
= \int_{-1}^1 Var_x q_n^y(y)\, dy + \int_{-1}^1 Var_y q_n^x(x)\, dx
+ \|q_n^x\|_{\infty} + \|q_n^y\|_{\infty},
\end{multline*}
because
 $$
 Var_x q_n^y(y) = \int_{-1}^1 |(q_n^y)_x(x, y)|\, dx,
 $$
 $$
 Var_y q_n^x(x) = \int_{-1}^1 |(q_n^x)_y(x, y)|\, dy.
 $$
Hence
$$
\|\partial_x \partial_y g_\delta\|_1
> \frac{1}{C_2 \delta} \Bigl(\|\partial_x^2 g_\delta\|_1
+ \|\partial_y^2 g_\delta\|_1 +
\|\partial_x g_\delta\|_{\infty}
+ \|\partial_y g_\delta\|_{\infty} \Bigr),
$$
which completes the justification of the first claim. Now the second one follows
by the closed graph theorem. Indeed, if there is no function with the desired properties,
then we obtain a linear operator $T$ from the space $E$ of Lipschitz functions
$f$ on the square $[-1,1]^2$ vanishing on the boundary and having
Sobolev repeated derivatives $\partial_x^2f$ and $\partial_y^2f$ in $L^1([-1,1]^2)$
to the space $L^1([-1,1]^2)$ defined by $Tf=\partial_x\partial_yf$, where
$\partial_x\partial_y$ is taken in the sense of distributions. The space $E$
is Banach with respect to the natural norm
$$
\|f\|_E=\|f\|_{\rm Lip}+\|\partial_x^2f\|_1+\|\partial_y^2f\|_1,
$$
where the Lipschitz norm $\|f\|_{\rm Lip}$ is defined by
$$
\|f\|_{\rm Lip}=\max_{[-1,1]^2} |f(x,y)|+L(f),
$$
and $L(f)$ is the minimal Lipschitz constant for $f$. The graph
of the operator $T$ is closed, which is seen, for example, from the fact
that $T$ is continuous
on $E$ with values in the space of distributions (or
in the negative Sobolev space $W^{2,-2}([-1,1]^2)$) and $L^1([-1,1]^2)$
is continuously embedded into the space of distributions (respectively,
into $W^{2,-2}([-1,1]^2)$). Similarly one can obtain a function $f$
of class~$C^1$: in the definition of $E$ we replace the class of Lipschitz functions
by the space $C^1([-1,1]^2)$ with its natural norm.
\end{proof}

We now show how a similar result can be deduced from Ornstein's example.
We are grateful to A.V.~Shaposhnikov for suggesting the following lemma.

\begin{lemma}
Let $B = B(0, 1)$ be the open unit ball in $\mathbb{R}^2$.
There is no number $C$ such that for every function $f \in C_{0}^{\infty}(B)$
one has
$$
\|\partial^2_{x} f\|_{1} \le
C (\| f\|_{\infty} +
\|\nabla f\|_{\infty} + \|\Delta f\|_1).
$$
\end{lemma}
\begin{proof}
Let us assume that such $C$
exists. Then for all $f\in C_{0}^{\infty}(B)$ we have
\begin{equation}\label{assumption}
\|\partial^2_{x} f\|_{1} \le
C (\| f\|_{\infty} + \|\nabla f\|_{\infty} + \|\Delta f\|_1).
\end{equation}
Let us fix $f \in C_{0}^{\infty}(B)$.
For any point $P = (P_x,  P_y) \in  B(0, 1)$ and any number $N \in \mathbb{N}$
we define $g_{P, N}$ as follows:
$$
g_{P, N} (x, y) := f(N(x - P_x), N(y - P_y)).
$$
It is easy to see that
$$
{\rm supp}(g_{P, N}) \subset B(P, 1/N),
$$
$$
\| g_{P, N} \|_{\infty} = \|f\|_{\infty},
\quad
\| \nabla  g_{P, N} \|_{\infty} = N \|f\|_{\infty},
$$
$$
\|\partial^2_{x} g_{P, N}\|_{1} =
\|\partial^2_{x} f\|_{1},
\quad
\|\Delta g_{P, N}\|_1 = \|\Delta f\|_1.
$$
Let us take $M= N^2/100$ disjoint balls
 $\{B(P_i, 1/N)\}_{i = 1}^{M}$ in $B(0, 1)$.
Let us define $g_N$ by the following formula:
$$
g_N := \sum_{i = 1}^{M} g_{P_i, N}.
$$
Then
$$
\|g_N\|_{\infty} = \| f\|_{\infty},
\quad
\| \nabla  g_{N} \|_{\infty} = N \|f\|_{\infty},
$$
$$
\|\partial^2_{x} g_{N}\|_{1} =
M\|\partial_{x}\partial_y f\|_{1},
\quad
\|\Delta g_{N}\|_1 = M\|\Delta f\|_1.
$$
Next we apply (\ref{assumption}) to the function $g_N$:
$$
M\|\partial^2_{x} f\|_{1}
\le
C(\|f\|_{\infty} + N \|\nabla f\|_{\infty}  + M\|\Delta f\|_1),
$$
which is
$$
\|\partial^2_{x} f\|_{1}
\le
C\Bigl(\frac{1}{M}\|f\|_{\infty} + \frac{N}{M} \|\nabla f\|_{\infty}  + \|\Delta f\|_1\Bigr).
$$
Letting $N \to \infty$ we obtain
$$
\|\partial^2_{x} f\|_{1}
\le C \|\Delta f\|_1.
$$
Now it is easy to see that since $f$ was an arbitrary
function in $C_{0}^{\infty}(B)$, this inequality holds
for every function $f \in C_{0}^{\infty}(\mathbb{R}^2)$.
This contradicts the result of Ornstein.
\end{proof}

We now prove an analog of Theorem~\ref{t2.1} (with the repeated derivative
in place of the mixed derivative).

\begin{theorem}
Let $B = B(0, 1)$ be the open unit ball in
$\mathbb{R}^2$.
There exists a Lipschitz function $f$ on $B$
such that
$$
\Delta f \in L^{1}(B), \ \partial^{2}_{x} f \notin L^{1}(B).
$$
\end{theorem}
\begin{proof}
Let $X$ be the completion
of $C_0^{\infty}(B)$ with respect to the norm
$$
\|f\|_X := \| f\|_{\infty} +
\|\nabla f\|_{\infty} + \|\Delta f\|_1.
$$
Let us assume that for each $f \in X$ we have
$\partial^{2}_{x} f \in L^{1}(B)$, where
$\partial^{2}_{x} f$ is understood in the sense of
distributions. Then by the closed graph theorem the operator $f\mapsto \partial_x^2f$
from $X$ to $L^1(B)$ is bounded. This contradicts the previous lemma.
\end{proof}

This result can be used in our main construction.


\vskip .2in

Vladimir I. Bogachev and Stanislav V. Shaposhnikov:
Department of Mechanics and Mathematics, Moscow State University, 119991 Moscow, Russia.

email: vibogach@mail.ru

and  National Research University Higher School of Economics, ul. Usacheva 6, 119048 Moscow, Russia

 Svetlana N. Popova:
Department of Mechanics and Mathematics, Moscow State University, 119991 Moscow, Russia

\end{document}